\documentclass[12pt,a4paper]{amsart}
\newtheorem*{theorem}{Theorem}
\newtheorem*{proposition}{Proposition}
\newtheorem{lemma}{Lemma}
\newtheorem*{corollary}{Corollary}
\newtheorem*{conjecture}{Conjecture}

\def\ralt{\mathcal{RA}lt}
\def\lalt{\mathcal{LA}lt}
\def\alt{\mathcal Alt}

\begin{document}

\title{The alternative operad is not Koszul}

\author{Askar Dzhumadil'daev}
\address{Kazakh-British Technical University, Tole bi 59, Almaty 050000, Kazakhstan}
\email{dzhuma@hotmail.com}

\author{Pasha Zusmanovich}
\address{Reykjavik Academy, Iceland}
\email{pasha@akademia.is}

\date{last revised May 4, 2011}
\thanks{\textsf{arXiv:0906.1272}; Experiment. Math. \textbf{20} (2011), 138--144}

\maketitle

In the online compendium \cite{loday}, it is asked whether the alternative operad is 
Koszul. The purpose of this note is to demonstrate that the answer to this question is 
negative. In doing so, we are helped with the programs Albert \cite{albert} and 
PARI/GP \cite{pari}.

\section{The alternative operad and its Koszul dual}\label{intro}

Recall that an algebra is called \textit{right-alternative} if it satisfies the identity
\begin{equation}\label{1}
(xy)y = x(yy) ,
\end{equation}
and \textit{left-alternative} if it satisfies the identity
\begin{equation}\label{2}
(xx)y = x(xy) .
\end{equation}
An algebra which is both right-alternative and left-alternative is called
\textit{alternative}. 

Linearizing identities (\ref{1}) and (\ref{2}), we get
\begin{equation}\label{ra}\tag{RA}
(xy)z + (xz)y - x(yz) - x(zy) = 0
\end{equation}
and
\begin{equation}\label{la}\tag{LA}
(xy)z + (yx)z - x(yz) - y(xz) = 0 ,
\end{equation}
respectively.
If the characteristic of the ground field is different from $2$, these identities are
equivalent to the initial ones, and they define binary quadratic operads 
$\ralt$, $\lalt$ and $\alt$ (dubbed \textit{right-alternative}, 
\textit{left-alternative} and \textit{alternative operads}).
In characteristic $2$ things go berserk: identities (\ref{1}) and (\ref{2}) 
are not equivalent to the corresponding linearized identities, so it is impossible 
to encode them with operads in a straightforward manner. We will exclude this case 
from our considerations.

Right- and left-alternative algebras are opposite to each other,
i.e., if $A$ is a right-alternative algebra, then the algebra defined on the same
vector space $A$ with multiplication $x \circ y = yx$ is a left-alternative algebra, 
and vice versa. Hence all the statements below for left-alternative
algebras automatically follow from the corresponding statements for right-alternative
ones, and in the proofs we will consider the right-alternative case only.
Most of these statements are trivial and/or have been considered previously in the 
literature, but they provide a good warm-up before the more difficult alternative case.

Every associative algebra is alternative. 
An example of a non-associative alternative algebra is the octonion 
algebra, appearing prominently in mathematics and physics (see, for example, the 
excellent survey \cite{baez}).
Note also that free alternative algebras are much more difficult objects than,
for example, their associative or Lie counterparts, and are still not understood 
sufficiently well.

For the general operadic business, including important notions of Koszulity and Koszul 
duality, we refer to the book \cite{mss-book} and the foundational paper \cite{gk}. 
However, to understand this note it is enough
to adopt an intuitive and primitive view on operads as polylinear parts of the corresponding free 
algebras, and to accept the Ginzburg--Kapranov criterion for Koszulity, as described below, 
for granted.

\begin{proposition}
Each of the operads Koszul dual to the right-alternative, left-alternative and alternative operad
is defined by two identities: associativity and the identity
\begin{equation}\tag{RA$^!$}\label{dra}
xyz + xzy = 0
\end{equation}
in the right-alternative case,
\begin{equation*}
xyz + yxz = 0
\end{equation*}
in the left-alternative case, and
\begin{equation}\label{nilp}\tag{A$^!$}
xyz + yxz + zxy + xzy + yzx + zyx = 0 .
\end{equation}
in the alternative case.
\end{proposition}

In the alternative case, this is stated in \cite{loday} without proof, 
so we will provide a simple (and pretty much standard for such situations) 
proof for completeness.
Following \cite{loday}, we will call algebras over the corresponding Koszul dual operads
\textit{dual right-alternative}, \textit{dual left-alternative}
and \textit{dual alternative}, respectively.

\begin{proof}
Let $R$ be the space of quadratic relations of the alternative operad, i.e.,
the space generated by the left-hand sides of identities (\ref{ra}) and (\ref{la}),
and $R^\perp$ be the space of quadratic relations of the dual alternative operad.

Identities (\ref{ra}) and (\ref{la}) imply that we may take the images of the $7$ monomials 
$(xy)z$, $(yx)z$, $(xz)y$, $(zx)y$, $(yz)x$, $(zy)x$, $z(xy)$
as the basis of $\alt(3)$, with the remaining monomials expressed through them as
follows:
\begin{align}
z(yx) &= \phantom{-} (zx)y + (zy)x - z(xy) \notag \\
y(zx) &=          -  (zx)y + (yz)x + z(xy) \notag \\
y(xz) &= \phantom{-} (yx)z + (zx)y - z(xy) \label{yoyo}\\
x(yz) &= \phantom{-} (xy)z - (zx)y + z(xy) \notag \\
x(zy) &= \phantom{-} (xz)y + (zx)y - z(xy) \notag .
\end{align}
In particular, 
$$
\dim \alt(3) = \dim R^\perp = 7
$$
and 
$$
\dim \alt^!(3) = \dim R = 3!\,C_2 - 7 = 5
$$ 
(here and below, $C_n = \frac{(2n)!}{n!(n+1)!}$ denotes the $n$th Catalan number).

To obtain identities defining the dual alternative operad, it is convenient to
use the fact that if $L$ is an alternative algebra, and $A$ is a dual alternative 
algebra, then their tensor product $L \otimes A$ equipped with the bracket 
$$
[x \otimes a, y \otimes b] = xy \otimes ab - yx \otimes ba
$$
for $x,y\in L$, $a,b\in A$, is a Lie algebra.
Writing the Jacobi identity for triple $x\otimes a$, $y\otimes b$, $z\otimes c$
for $x,y,z\in L$, $a,b,c\in A$, substituting in it all equalities (\ref{yoyo}), 
and collecting similar terms, we get:
\begin{align*}
  & (xy)z \otimes ((ab)c - a(bc))   \\
+\> & (yx)z \otimes (b(ac) - (ba)c)   \\
+\> & (xz)y \otimes (a(cb) - (ac)b)   \\
+\> & (zx)y \otimes (a(bc) + a(cb) + b(ac) + b(ca) + (ca)b + c(ba))  \\
+\> & (yz)x \otimes ((bc)a - b(ca))   \\
+\> & (zy)x \otimes (c(ba) - (cb)a)   \\
-\> & z(xy) \otimes (a(bc) + a(cb) + b(ac) + b(ca) + c(ab) + c(ba))  \\ 
=\> & 0 ,
\end{align*}
and the claimed identities follow.

Now it is straightforward to check that the so obtained relations are orthogonal to 
the alternative relations $R$ under the standard pairing 
(as defined in \cite[\S 2.1.11]{gk}), so they really lie in $R^\perp$. 
Under the action of the symmetric group $S_3$, the associativity gives
$6$ different relations, and the left-hand side of (\ref{nilp}) is $S_3$-invariant,
so we get $7$ relations in total. This shows that all relations are taken into 
account.

In the right-alternative case we have $\dim \ralt(3) = 9$, and the computations are similar.
\end{proof}

\begin{corollary}
\noindent
\begin{enumerate}
\item 
A dual right- or left-alternative algebra over a field of characteristic different from $2$,
is nilpotent of degree $4$.
\item 
A dual alternative algebra over a field of characteristic different from $2$ and $3$,
is nilpotent of degree $6$.
\end{enumerate}
\end{corollary}

\begin{proof}
(i)
We have, by subsequent application of associativity and (\ref{dra}):
$$
(xyz)t = -(xzy)t = -x(zyt) = x(zty) = x(zt)y = - xy(zt).
$$

(ii)
Substituting $x=y=z$ in (\ref{nilp}), we get $6x^3 = 0$.
The claim then follows from the results centered around the classical
Dubnov--Ivanov--Nagata--Higman Theorem about nilpotency of associative nil algebras
(see, for example, \cite[\S 8.3]{drensky}). 

These claims also could be proved with the help of Albert.
\end{proof}

\section{Dimension sequence}\label{dim}

We are going to establish non-Koszulity using the well-known Ginzburg--Kapranov criterion 
\cite[Proposition 4.1.4(b)]{gk} which tells that if a finitely generated binary quadratic operad $\mathcal P$
over a field of characteristic zero is Koszul, then 
\begin{equation}\label{gk}
g_{\mathcal P} (g_{\mathcal P^!} (x)) = x ,
\end{equation}
where 
$$
g_{\mathcal P}(x) = \sum_{n=1}^\infty (-1)^n \frac{\dim \mathcal P(n)}{n!} x^n
$$ 
is the Poincar\'e series of the operad $\mathcal P$, and $\mathcal P^!$
is the Koszul dual of $\mathcal P$.
For this, we need to know the first few terms of the sequence $\dim \mathcal P(n)$
for the corresponding operads and/or their Koszul duals.
This is achieved with the help of Albert.

Albert computes over a fixed prime field, and we are going to explain now 
how these computations imply results valid in characteristic zero. 

Representing an operad $\mathcal P$ as the quotient of the free (= magmatic)
operad $\mathcal F$ by the ideal of relations, and considering the corresponding 
arity $n$ parts for a fixed $n$, we have 
$$
\dim \mathcal P(n) + rk\, M = \dim \mathcal F(n) = n!\,C_{n-1} ,
$$
where $M$ is a matrix consisting of coefficients of all linear
relations in $\mathcal P$ between all nonassociative multilinear monomials in $n$ 
variables.
As coefficients of identities defining our operads
are integers, $M$ is an integer matrix, and it is possible to consider
its reduction $M_p$ modulo a given prime $p$. 

It is clear that $rk\, M \ge rk\, M_p$. The question is how to ensure equality of these 
values. 
What follows is a variation on the standard theme in numerical linear 
algebra -- how to substitute rational or integer arithmetic by modular one.

Let represent the matrix $M$ in the Smith normal form, i.e., as a product
$$
M = S \> diag(d_1, \dots, d_r, 0, \dots, 0) \> T ,
$$
where $S$ and $T$ are integer quadratic matrices with determinant equal to $\pm 1$, 
$r = rk\, M$, and 
$d_1, \dots, d_r$ are nonzero integers such that $d_i$ is divided by $d_{i+1}$.
Reduction of this product modulo $p$ will produce the Smith normal form of $M_p$,
i.e., $S_p$, $T_p$ are still matrices with determinant $\pm 1$ over 
$\mathbb Z/p\mathbb Z$,
and the number of nonzero elements in the diagonal matrix
$$
diag (d_1(mod\,p), \dots, d_r(mod\,p), 0, \dots, 0)
$$
is equal to $rk\,M_p$.

If we pick primes $p_1, \dots, p_k$ in such a way that 
\begin{equation}\label{ineq}
p_1 \dots p_k > |d_1 \dots d_r| ,
\end{equation}
then 
$$
d_1 \dots d_r \not\equiv 0 \,(mod\, p_1 \dots p_k) ,
$$
hence by the Chinese Remainder Theorem
$$
d_1 \dots d_r \not\equiv 0 \,(mod\, p_i) ,
$$
and hence $rk\,M_{p_i} = rk\,M$ for some $p_i$. Consequently, if $rk\,M_{p_i} = r$
for \textit{all} $i = 1, \dots, k$, then $rk\,M = r$.

It remains to estimate $p_1 \dots p_k$ to ensure inequality (\ref{ineq}).
The product $d_1 \dots d_r$ is equal, up to sign, to the determinant of a certain 
minor $Q$ of $M$ of size $r \times r$.
As the identities defining our operads have coefficients $1$ or $-1$, all nonzero
elements of the matrix $M$ could be chosen to be equal to $1$ or $-1$, so the usual estimate in such 
situations is provided by the Hadamard inequality: 
$|\det(Q)| \le r^{\frac{r}{2}}$ (see, for example, \cite[\S 7.8.2]{horn-j}). 

To summarize: if there are primes $p_1, \dots, p_k$ such that
Albert produces the same value 
\begin{equation}\label{pn}
\dim \mathcal P(n) = m
\end{equation}
modulo these primes, and
\begin{equation}\label{est}
p_1 \dots p_k > r^{\frac{r}{2}} , \quad\text{ where }\> r = n!\,C_{n-1} - m ,
\end{equation}
then (\ref{pn}) holds over integers, and, consequently, over any field
of characteristic zero. 

Albert allows to compute over a prime field $\mathbb Z/p\mathbb Z$ with $p \le 251$.
We have modified Albert \cite{albert-mod} to allow primes up to the largest possible value
of the largest signed integer type, which is $2^{63} - 1$ on the standard modern 
computer architectures, both 32-bit and 64-bit.
We also have modified it to facilitate batch processing. 

As the time of Albert computations turns out not to depend significantly on the value of prime,
to minimize the overall computation time,
we are minimizing the number of Albert runs at the expense of larger primes, 
i.e., when choosing primes in the given range satisfying the condition (\ref{est}),
we are choosing as large primes as possible. 
This could be done with the help of PARI/GP.

Using all this, we establish:

\begin{lemma}\label{lemma-ralt}
Over a field of characteristic zero, the first $5$ terms of the sequence $\dim \ralt(n)$
are: $1, \> 2, \> 9, \> 60, \> 530$.
\end{lemma}

\begin{proof}
Over any field, the first two values are obvious, and the third could be established 
by hand (in fact, we already did it in the proof of Proposition in \S \ref{intro}).

The are $3$ primes $< 2^{63}$ satisfying the inequality (\ref{est}) 
for $r = 4!\,C_3 -60 = 60$:
$$
2^{63} - 259,  \quad
2^{63} - 165,  \quad
2^{63} - 25.
$$
With all these $3$ primes, Albert produces $\dim \ralt(4) = 60$.

Similarly, the number of largest possible primes $< 2^{63}$ satisfying the 
inequality (\ref{est}) for $r = 5!\,C_4 - 530 = 1150$, is $93$, and Albert produces $\dim \ralt(5) = 530$
for all these $93$ primes.
\end{proof}

We have also computed $\dim \ralt(6) = 5820$ for a few random primes\footnote[2]{
The corresponding sequence was submitted to \cite{eis} as A161391.
}.

\begin{lemma}\label{lemma-alt}
Over a field of characteristic zero, the first $6$ terms of the sequence $\dim \alt(n)$
are: $1, \> 2, \> 7, \> 32, \> 175, \> 1080$\footnote[3]{
This sequence was submitted to \cite{eis} as A161392.
}.
\end{lemma}

\begin{proof}
We follow the same scheme as in the proof of Lemma \ref{lemma-ralt}.
The corresponding number of primes is $5$ for $n=4$, $127$ for $n=5$, 
and $3433$ for $n=6$, and Albert produces the expected answers for all these primes.
\end{proof}

The first $5$ terms in Lemma \ref{lemma-alt} were already specified in \cite{loday}, 
but the case $n=6$ is crucial.
It requires the only time-consuming Albert computations among all computations 
mentioned in this note.
We found that the optimal setting in this case was to add first the left-alternative
identity, and then the right-alternative one, and use the static (as opposed to the 
sparse) matrix structure. The whole computation was finished in about a week running 
in parallel on a number of CPUs ranging from 2GHz single-core to 3.2GHz dual-core.
The average execution time was less than 1 hour per prime.

\section{Non-Koszulity}\label{nonk}

\begin{theorem}\label{main}
The right-alternative, left-alternative and alternative operads over a field of 
characteristic zero are not Koszul.
\end{theorem}

\begin{proof}
The statement for the right(left)-alternative case is known (and easy),
but it will be instructive to look on it first and to compare it with the more difficult 
alternative case.

By Proposition and Corollary (i) in \S \ref{intro}, $\dim \ralt^!(3) = 3$ and 
$\ralt^!(n)$ vanishes for $n\ge 4$, so the corresponding Poincar\'e series is:
\begin{equation*}
g_{\ralt^!} (x) = - x + x^2 - \frac{1}{2} x^3 .
\end{equation*}
On the other hand, by Lemma \ref{lemma-ralt},
$$
g_{\ralt} (x) = - x + x^2 - \frac{3}{2} x^3 + \frac{5}{2} x^4 - \frac{53}{12} x^5 + O(x^6)
,
$$
and
$$
g_{\ralt} (g_{\ralt^!} (x)) = x + \frac{1}{6} x^5 + O(x^6) ,
$$
what contradicts Koszulity.

But, in fact, we can establish the same without appealing to Lemma \ref{lemma-ralt}! 
Indeed, the beginning terms of the inverse to the polynomial $g_{\ralt^!}(x)$ are:
$$
- x + x^2 - \frac{3}{2} x^3 + \frac{5}{2} x^4 - \frac{17}{4} x^5 + 7 x^6 
- \frac{21}{2} x^7 + \frac{99}{8} x^8 - \frac{55}{16} x^9 - \frac{715}{16} x^{10} 
+ O(x^{11}) .
$$
The signs alternation is violated at the $10$th term, hence this series cannot be
the Poincar\'e series of any operad, so by the Ginzburg--Kapranov criterion $\ralt^!$
is not Koszul, and hence $\ralt$ is not Koszul.
Moreover, the dimension sequence of $\ralt^!$ coincides with the dimension sequence
of the operad $\mathcal Prelie \bullet \mathcal Nil$ ($\mathcal Prelie$ is the 
operad defined by a binary operation satisfying the pre-Lie (=right symmetric) identity, 
$\mathcal Nil$ is the operad defined by a skew-symmetric binary 
operation with vanishing compositions, and $\bullet$ is Manin's black product), and the 
corresponding computation establishing its non-Koszulity was already performed in 
\cite[\S 4.5]{vallette}.

Now consider the alternative case.
By Corollary (ii) in \S \ref{intro}, $\mathcal Alt^!(n)$ vanishes for $n \ge 6$.
Either computation with Albert, or reference to \cite[Propositions 1 and 2]{lopatin} 
provides dimensions of these spaces for small $n$, which allows us to write down 
the Poincar\'e series of the operad $\mathcal Alt^!$:
\begin{equation}\label{galt}
g_{\mathcal Alt^!} (x) 
= - x + x^2 - \frac{5}{6} x^3 + \frac{1}{2} x^4 - \frac{1}{8} x^5 .
\end{equation}
On the other hand, by Lemma \ref{lemma-alt},
\begin{equation*}
g_{\mathcal Alt} (x) 
= 
- x + x^2 - \frac{7}{6} x^3 + \frac{4}{3} x^4 - \frac{35}{24} x^5 + \frac{3}{2} x^6 
+ O(x^7) ,
\end{equation*}
and
$$
g_{\mathcal Alt} (g_{\mathcal Alt^!} (x)) = x - \frac{11}{72} x^6 + O(x^7) ,
$$
what contradicts Koszulity.

Note that in the alternative case we really need to compute dimension sequence of the 
alternative operad up to $6$th term (i.e., to utilize Lemma \ref{lemma-alt}). A mere look at the 
inverse to the polynomial $g_{\mathcal Alt^!}(x)$  does not seem to work: we have checked with PARI/GP that
the inverse has alternating signs up to degree $1000$.
On the other hand, as noted in \cite[\S 4.2]{goze-remm-arxiv}, the beginning terms 
of the inverse to $g_{\mathcal Alt}(x)$ are:
$$
-x + x^2 - \frac{5}{6} x^3 + \frac{1}{2} x^4 - \frac{1}{8} x^5 - \frac{11}{72} x^6 
+ O(x^7) ,
$$
what provides an alternative proof of non-Koszulity of $\alt$ without appealing to 
$g_{\mathcal Alt^!}(x)$.
\end{proof}

Sometimes in the literature one sees expressed the viewpoint that non-Koszulity is a 
rather pathological property, and all ``natural'', ``occuring in the real life'' 
algebras should be algebras over a Koszul operad (see, for example, 
Remarks 3.98 and 3.131 in \cite{mss-book}). 
As we see, alternative algebras provide a ``real life'' example
violating this principle (another, albeit probably less ``real life'' contender is 
presented in \cite{dzhu-nov}).

\section{Positive characteristic}

The original Ginzburg--Kapranov operadic theory involves representation theory of the
symmetric group peculiar to characteristic zero case. 
While extensions of the operadic theory to the case of 
positive characteristic exist, none of them, to our knowledge, includes 
an analog of the Ginzburg--Kapranov criterion for Koszulity of a quadratic operad 
in terms of Poincar\'e series.

While, therefore, checking the validity of equation (\ref{gk}) in positive characteristic 
does not make much sense, the question of computing the dimension sequence 
$\dim \mathcal P(n)$ for various operads $\mathcal P$ is still of interest.
In this section we collect a few remarks and computational results concerning
this question for the alternative and right-alternative operads and their Koszul duals.

For the Koszul dual operads, the corresponding dimension sequences terminate at low terms
as indicated in the proof of the theorem in \S \ref{nonk}, the same way for zero and positive
characteristics, except for the case of the dual alternative operad over a field
of characteristic~$3$.

\begin{conjecture}
Over a field of characteristic $3$, $\dim \mathcal Alt^! (n) = 2^n - n$.
\end{conjecture}

For $n \le 8$ the claim could be proved with the aid of Albert.
We will outline the main idea of a possible proof in the general case, whose full 
implementation appears to be long and somewhat cumbersome, and will drive us far away
from the main question considered in this note. We came up with this idea by inspecting the 
corresponding entry A000325 in \cite{eis}.

\begin{proof}[Sketch of a possible proof]
For associative algebras, the identity (\ref{nilp}) is equivalent to the identity
\begin{equation*}
[[x,y],y] = 0 .
\end{equation*}
In other words, an associative algebra over a field of characteristic $3$ is 
dual alternative if and only if its associated Lie algebra is $2$-Engel. It is well-known
that $2$-Engel Lie algebras are nilpotent of order $4$. Free associative algebras
which are Lie-nilpotent of order $4$ were studied in the recent paper \cite{etingof}.
It is possible to extend some of the results of that paper to the case of characteristic $3$,
and, in particular, to construct a presentation of such algebras. From this, by adding
more relations, one may construct a presentation of free dual alternative algebras,
and using Composition (=Diamond) Lemma, to get a description of a basis of such algebras
in combinatorial terms. For elements of the basis containing each free generator
in the first degree, these combinatorial terms are expressed as the so-called Grassmann
permutations, i.e. $\alt^!(n)$ has a basis consisting of associative monomials of the form
$a_{i_1} \dots a_{i_n}$ such that the permutation $(i_1 \dots i_n)$ has exactly one
descent. The number of such permutations is $2^n - n$.
\end{proof}

The case of characteristic $3$ is also exceptional for the alternative operad:
in this case, the first $5$ terms of $\dim \mathcal Alt (n)$ are the same as in 
Lemma \ref{lemma-alt}, while the $6$th term is equal, surprisingly, to 
$1081$\footnote[2]{
The corresponding sequence was submitted to \cite{eis} as A161393.
}.

Note also that the scheme of computations presented in \S \ref{dim} is insufficient
to deduce the validity of (\ref{pn}) over \textit{all} prime fields. 
Either by the standard ultraproduct argument, or observing,
by the same argument as in \S 2, that the equality (\ref{pn}) in characteristic
zero implies the same equality in characteristic $p$ for all $p > r^{\frac{r}{2}}$,
we may deduce that it is valid for all but a finite number of characteristics.
So, in principle, we could establish the validity of (\ref{pn}) in all characteristics 
by verifying it modulo all primes $\le r^{\frac{r}{2}}$ and for one prime 
$> r^{\frac{r}{2}}$. 
This is, however, computationally infeasible in almost all practical cases.

To be able to establish the equality (\ref{pn}) in all characteristics, apparently other
methods are needed. For example, one may try to use the capability of Albert
to produce multiplication table between elements of $\mathcal P(n)$ up to the given
degree. It seems that the scheme, based on the Chinese Remainder Theorem and 
similar to those presented in \S 2, but utilizing the multiplication table instead of
just dimensions of the corresponding spaces of multilinear monomials,
could be used for that, provided that all coefficients in the computed multiplication
tables are rational numbers with relatively small numerators and denominators modulo
the respective primes.
According to a few Albert computations we have performed for 
$\alt(6)$, the latter seems to be the case for the alternative operad.

\section{Questions}

In addition for an already mentioned in \S \ref{nonk} example from \cite{vallette}, 
there are several other proofs in the literature of non-Koszulity of various operads 
using the Ginzburg--Kapranov criterion or its $n$-ary analogs: in \cite[footnote to \S 3.9(d)]{getzler-k} 
for the so-called mock-Lie and mock commutative operads
(which are Koszul dual to each other and are cyclic quadratic operads with one 
generator); in \cite[Remark 3.98]{mss-book} for associative anticommutative algebras
(and, dually, for ``commutative Lie algebras''); 
in \cite[Proposition 2.3]{goze-remm} for certain Lie-admissible
operads dubbed $G_4$-$Ass$ and $G_5$-$Ass$;
in \cite[\S 3.4,3.6]{goze-remm-arxiv} for certain third power associative operads dubbed $G_i$-$p^3Ass$;
in \cite[Theorem 10.1]{dzhu-alia} for a certain skew-symmetric operad dubbed 
left-Alia; in \cite{dzhu-nov} for the Novikov operad;
and in \cite[Example 16 and Proposition 17]{markl-remm} for certain operads with 
$n$-ary operation dubbed $tAss^n_d$.
In each of these cases, it was enough to check Poincar\'e series up to a relatively
low degree term.
It is interesting whether there exists a bound on the degree of Poincar\'e series
such that the validity of the identity (\ref{gk})
for a binary quadratic operad $\mathcal P$
up to this degree guarantees its validity in all degrees.

It is also interesting to give a concrete example of a binary quadratic operad 
which is not Koszul but for which the equality (\ref{gk}) holds (such examples exist for 
associative quadratic algebras -- see \cite[\S 3.5]{pp} and references therein).

Is it true that all terms of the inverse of the polynomial (\ref{galt}) have 
alternating signs? If yes, what combinatorial interpretation this may have?
(Question asked by Vladimir Dotsenko). A similar question
about an innocently-looking polynomial of degree $15$ and with only $3$ nonzero terms
was asked in \cite{markl-remm}. Somewhat surprisingly, such questions seem to be difficult.

Note also that it remains a challenging problem to compute the 
Poincar\'e series of $\mathcal Alt$.

And, finally, we are taking the opportunity to advertise some new classes of algebras.
In \cite[Theorem 5.1]{dzhu-alia}, all possible skew-symmetric identities of degree $3$ 
are classified. This classification has a symmetric analog: namely, every symmetric
identity of degree $3$ could be reduced to one of the following identities:
\begin{gather*}
[\{x,y\},z]+[\{y,z\},x]+[\{z,x\},y] = 0            \\
\{\{x,y\},z\}+\{\{y,z\},x\}+\{\{z,x\},y\} = 0      \\
\{x,y\}z+\{y,z\}x+\{z,x\}y = 0 ,
\end{gather*}
where $[x,y] = xy - yx$ and $\{x,y\} = xy + yx$.
Any right- or left-alternative algebra satisfies the first of these identities,
and the second identity is exactly (\ref{nilp}) (with appropriately inserted left-normed
brackets, as associativity is no longer assumed).
It appears to be interesting to study algebras satisfying these identities,
in particular, describe free and simple algebras in these classes, and to look at the
corresponding operads.

\section*{Acknowledgements}

Thanks are due to Vladimir Dotsenko and Leonid Positselski for explanations
concerning Koszulity and the Ginzburg--Kapranov criterion, to Alexander Feldman for 
advice concerning computing issues and comments on the preliminary version of the 
manuscript, and to the anonymous referee for a very careful reading of the manuscript
and useful comments as well.

\renewcommand{\refname}{Software and online repositories}


\begin{thebibliography}{EKM}

\bibitem[B]{baez} J. Baez, \emph{The octonions}, 
Bull. Amer. Math. Soc. \textbf{39} (2002), 145--205; Errata: \textbf{42} (2005), 213;
\textsf{arXiv:math/0105155}.

\bibitem[Dr]{drensky} V. Drensky, \emph{Free Algebras and PI-Algebras}, Springer, 2000.

\bibitem[Dz1]{dzhu-alia} A.S. Dzhumadil'daev, 
\emph{Algebras with skew-symmetric identity of degree $3$}, 
J. Math. Sci. \textbf{161} (2009), 11--30.

\bibitem[Dz2]{dzhu-nov} \bysame, 
\emph{Codimension growth and non-Koszulity of Novikov operad}, Comm. Algebra, to appear;
\textsf{arXiv:0902.3771v1}.

\bibitem[EKM]{etingof} P. Etingof, J. Kim and X. Ma, 
\emph{On universal Lie nilpotent associative algebras}, 
J. Algebra \textbf{321} (2009), 697--703; \textsf{arXiv:0805.1909}. 

\bibitem[GeK]{getzler-k} E. Getzler and M. Kapranov, 
\emph{Cyclic operads and cyclic homology}, 
Geometry, Topology and Physics for Raoul Bott (ed. S.-T. Yau), 
Conf. Proc. Lect. Notes Geom. Topol. \textbf{4} (1995), 167--201.

\bibitem[GiK]{gk} V. Ginzburg and M. Kapranov, \emph{Koszul duality for operads},
Duke Math. J. \textbf{76} (1994), 203--272; Erratum: \textbf{80} (1995), 293;
\textsf{arXiv:0709.1228}.

\bibitem[GR1]{goze-remm} M. Goze and E. Remm, \emph{Lie-admissible algebras and operads},
J. Algebra \textbf{273} (2004), 129--152; \textsf{arXiv:math/0210291}.

\bibitem[GR2]{goze-remm-arxiv} \bysame{} and \bysame, 
\emph{A class of nonassociative algebras including flexible and alternative algebras, 
operads and deformations}; \textsf{arXiv:0910.0700v1}.

\bibitem[HJ]{horn-j} R. Horn and C. Johnson, \emph{Matrix Analysis}, 
Cambridge Univ. Press, 1990.

\bibitem[Lod]{loday} J.-L. Loday, \emph{Encyclopedia of types of algebras}, June 2007;
\newline
\texttt{http://www-irma.u-strasbg.fr/$\sim$loday/jllpub.html} 

\bibitem[Lop]{lopatin} A.A. Lopatin, 
\emph{Relatively free algebras with the identity $x^3 = 0$},
Comm. Algebra \textbf{33} (2005), 3583--3605; \textsf{arXiv:math/0606519}.

\bibitem[MR]{markl-remm} M. Markl and E. Remm, 
\emph{(Non-)Koszulity of operads for $n$-ary algebras, cohomology and deformations};
\textsf{arXiv:0907.1505v1}.

\bibitem[MSS]{mss-book} \bysame, S. Shnider and J. Stasheff, 
\emph{Operads in Algebra, Topology and Physics}, AMS, 2002.

\bibitem[PP]{pp} A. Polishchuk and L. Positselski, \emph{Quadratic Algebras}, 
AMS, 2005.

\bibitem[V]{vallette} B. Vallette, 
\emph{Manin products, Koszul duality, Loday algebras and Deligne conjecture},
J. Reine Angew. Math. \textbf{620} (2008), 105--164; \textsf{arXiv:math/0609002}.

\end{thebibliography}

\begin{thebibliography}{OEIS}

\bibitem[A1]{albert} Albert version 4.0;
\texttt{http://www.cs.clemson.edu/$\sim$dpj/albertstuff/albert.html} 

\bibitem[A2]{albert-mod} Albert version 4.0M (modified);
\texttt{http://justpasha.org/math/albert/}

\bibitem[OEIS]{eis} The On-Line Encyclopedia of Integer Sequences;
\texttt{http://www.research.att.com/$\sim$njas/sequences/}

\bibitem[P]{pari} PARI/GP; \texttt{http://pari.math.u-bordeaux.fr/} 

\end{thebibliography}
\end{document}


\title{Corrigendum to ``The alternative operad is not Koszul''}

\author{Askar Dzhumadil'daev}
\address{Kazakh-British Technical University, Almaty, Kazakhstan}
\email{dzhuma@hotmail.com}

\author{Pasha Zusmanovich}
\address{Tallinn University of Technology, Tallinn, Estonia}
\email{pasha.zusmanovich@ttu.ee}

\date{September 9, 2012; last revised November 6, 2012}
\thanks{Experiment. Math., to appear}

\maketitle

In \cite[\S 4]{alternative}, we formulated a conjecture that in characteristic 
$3$, the dimension of the $n$th homogeneous component of the dual alternative 
operad, i.e. an operad governed by two identities -- associativity and
\begin{equation}\label{eq}\tag{*}
xyz + yxz + zxy + xzy + yzx + zyx = 0
\end{equation}
(or, what is the same, dimension of the multilinear 
component of the corresponding free algebra), is equal to $2^n - n$.

In fact, this was proved earlier by Lopatin 
(see \cite[\S 7, Remark 2]{lopatin}): he computes the corresponding dimension
for the variety of associative algebras satisfying the identity $x^3 = 0$,
what for multilinear components is equivalent to the corresponding dimensions
of its full linearization (\ref{eq}). Lopatin's proof consists of computer 
calculations for small values of $n$ (as we did in \cite{alternative}), and an
argument based on the composition (=diamond) lemma reducing the general case 
to the cases of small $n$'s.

Thanks are due to Ivan Kaygorodov for bringing this fact to our attention, 
and to Artem Lopatin for explaining some points of \cite{lopatin}.

Recently, a more general result was established by Dotsenko in \cite{dotsenko}.
Dotsenko's proof utilizes a generalization of composition lemma for operads, 
and does not depend on any computer calculations.